\documentclass[12pt]{amsart}
\usepackage[english]{babel}
\usepackage{paralist,url,verbatim, anysize}
\usepackage{amscd}
\usepackage[all,cmtip]{xy} 
\usepackage{amsmath,amsthm,amssymb,amsfonts, mathtools}
\usepackage[bookmarksopen=true]{hyperref}
\usepackage[usenames, dvipsnames]{color}
\usepackage{graphicx}
\usepackage{hyperref}
\usepackage{euler, amsfonts, amssymb, latexsym, epsfig,epic}
\usepackage{stmaryrd}

\usepackage{tikz}
\usetikzlibrary{decorations.markings,intersections,positioning,calc}

  \tikzset{mylabel/.style  args={at #1 #2  with #3}{
    postaction={decorate,
    decoration={
      markings,
      mark= at position #1
      with  \node [#2] {#3};
 } } } }

\makeatletter
\let\@fnsymbol\@arabic
\makeatother

\theoremstyle{plain}
\newtheorem{theorem}{\bf Theorem}[section]

\newtheorem{corollary}[theorem]{Corollary}
\newtheorem{lemma}[theorem]{Lemma}

\newtheorem{theoremx}{Theorem}
\newtheorem{questionx}[theoremx]{Question}

\theoremstyle{definition}

\newtheorem{definition}[theorem]{Definition}
\newtheorem{notation}[theorem]{Notation}

\newtheorem{remark}[theorem]{Remark}

\newtheorem*{theorem*}{\bf Theorem}

\newcommand{\init}{\operatorname{in} }

\newcommand{\R}{\mathbb{R}}

\newcommand{\Min}{\operatorname{Min} }

\newcommand{\Spec}{\operatorname{Spec} }

\newcommand{\ord}{\mathrm{ord}}
\newcommand{\vars}{\mathrm{vars}}

\newcommand{\fpt}{\mathrm{fpt}}
\newcommand{\dfpt}{\mathrm{dfpt}}

\newcommand{\IN}{\mathrm{in}}
\newcommand{\p}{\mathfrak{p}}

\newcommand{\m}{\mathfrak{m}}
\newcommand{\n}{\mathfrak{n}}
\renewcommand{\a}{\mathfrak{a}}

\newcommand{\Hom}{\operatorname{Hom} }

\definecolor{mypink}{RGB}{215, 5, 234}

 \newcommand{\lcm}{\operatorname{lcm}}

\newcommand{\supp}{\operatorname{supp}}

\newcommand{\ZZ}{\ensuremath{\mathbb{Z}}}

\newcommand{\e}{\operatorname{e}}

\begin{document}
\title{F-singularities of polynomials with square-free support}
\author[A. Conca]{Aldo Conca} 
\address{Dipartimento di Matematica, Universit{\`a} di Genova, Via Dodecaneso 35, 16146 Genova, Italy}
\email{aldo.conca@unige.it}

\author[A. De Stefani]{Alessandro De Stefani}
\address{Dipartimento di Matematica, Universit{\`a} di Genova, Via Dodecaneso 35, 16146 Genova, Italy}
\email{alessandro.destefani@unige.it}

\author[L. N\'u\~nez-Betancourt]{Luis N\'u\~nez-Betancourt}
\address{Centro de Investigaci\'on en Matem\'aticas, Guanajuato, Gto., M\'exico}
\email{luisnub@cimat.mx}

\author[I. Smirnov]{Ilya Smirnov}
\address{BCAM -- Basque Center for Applied Mathematics, Mazarredo 14, 48009 Bilbao, Basque Country -- Spain}
\address{Ikerbasque, Basque Foundation for Science, Plaza Euskadi 5, 48009 Bilbao, Basque Country -- Spain}
\email{ismirnov@bcamath.org}

  \date{}

\begin{abstract}
We show that the intersection of the irreducible components of a hypersurface defined by a polynomial with square-free support has F-rational singularities in characteristic $p>0$. In particular, we obtain that hypersurfaces defined by irreducible polynomials with square-free support have F-rational singularities, positively answering a question of Bath, Musta\c{t}\u{a}, and Walther. 
\end{abstract}

\maketitle

\section{Introduction}

We say that a polynomial $f$ over a field $K$ is \emph{square-free supported} if every variable appearing in $f$ has degree at most one, or, equivalently, if every monomial appearing in $f$ is square-free. 
The interest in square-free supported polynomials comes from matroid support polynomials or, more generally, matroidal polynomials \cite{BW}. 
Recently Bath, Musta\c{t}\u{a}, and Walther \cite{BMW} proved that 
if $K$ is an algebraically closed field of characteristic zero
and $f$ is an irreducible square-free supported polynomial in $n$ variables, then the  hypersurface $Z \subseteq \mathbb{A}^n_K$ cut out by $f$ has rational singularities. 
The celebrated connection between F-rational and rational singularities, 
due to the work of Smith \cite{Smith}  and Hara \cite{Hara} (see also \cite{MehtaSrinivas}), makes the following question quite natural.

\begin{questionx}[{\cite[Question 1.2]{BMW}}] \label{Quest} If $K$ is an algebraically closed field of characteristic $p>0$, and $Z \subseteq \mathbb{A}^n_K$ is the hypersurface defined by an irreducible square-free supported polynomial, does $Z$ have F-rational singularities?
\end{questionx}

Our main theorem provides a positive answer to Question~\ref{Quest}. In fact, we show a more general result that applies for certain classes of complete intersections. Moreover, we remove the assumption that $K$ is algebraically closed. 

\begin{theoremx}[{Theorem~\ref{ThmReducible}}] \label{mainthm} 
Let $K$ be a field of characteristic $p > 0$ and $S=K[x_1,\ldots,x_n]$. Let $f \in S$ be a square-free supported polynomial with irreducible factors $f_1, \ldots,  f_t$. Then $S/(f_1, \ldots, f_t)$ is F-rational.
\end{theoremx}

We use Theorem~\ref{mainthm} to also give a positive characteristic version 
of \cite[Theorem 1.3]{BMW}, itself generalizing a result on Feynman diagram polynomials \cite[Theorem~6.44]{BW}.

\begin{theoremx}[{Corollary~\ref{CorSqfreeModification}}] \label{notmainthm}
Let $K$ be a field of positive characteristic, $S=K[x_1,\ldots,x_n]$, and 
$g,h \in S$ be square-free supported homogeneous polynomials. 
If $g$ is irreducible and does not divide $h$, then
for every $a_1, \ldots, a_n \in K$ the quotient 
$S/(g(1 + a_1 x_1 + \cdots + a_nx_n) + h)$ is F-rational.
\end{theoremx}

Through the technique of reduction modulo $p$, our results recover and improve on the results of Bath, Musta\c{t}\u{a}, and Walther \cite{BMW}. Compared to \cite{BMW}, Theorems~\ref{mainthm},\ref{notmainthm} allow to remove the assumption that the base field is algebraically closed also in characteristic zero (see Corollaries~\ref{Char0}, \ref{CorSqfreeModificationRational}).

The second part of our paper explores the singularities of square-free supported polynomials using the \emph{defect of the F-pure threshold}  (see Definition \ref{DefDFPT}), a numerical invariant recently defined to study the geometry of F-finite F-split schemes~\cite{DSNBS}. We compute this invariant for rings associated to square-free supported polynomials.

\begin{theoremx}[{Theorem \ref{ThmDfptOne} \& Corollary \ref{CorDfptMany}}] \label{main dfpt} Let $K$ be an algebraically closed field of characteristic $p>0$, and $S=K[x_1,\ldots,x_n]$. Let $f \in S$ be a square-free supported polynomial with irreducible factors $f_1, \ldots,  f_t$. Then,
\[
\dfpt(S/(f)) = \e(S/(f))-1
\quad \& \quad
\dfpt(S/(f_1, \ldots,  f_t)) = \e(S/(f))-t.
\]
\end{theoremx}

\subsection*{Acknowledgments}
 The first and second authors were partially supported by the MIUR Excellence Department Project CUP D33C23001110001, and by INdAM-GNSAGA. The third author thanks CONAHCYT, Mexico, for its support with Grants CBF 2023-2024-224 and CF-2023-G-33.
The fourth author was supported by the
Ramon y Cajal Fellowship RYC2020-028976-I, funded by MCIN/AEI/10.13039/501100011033 and by FSE ``invest in your future'',
 PID2021-125052NA-I00, funded by MCIN/AEI/10.13039/501100011033, and 
 EUR2023-143443, funded by MCIN/AEI/10.13039/501100011033 and the European Union NextGenerationEU/PRTR.

\section{Strong F-regularity of square-free supported polynomials}

\begin{notation}
Let $S=K[x_1,\ldots,x_n]$, and $f\in S$. We let $\supp(f)$ be the support of $f$, that is, the set of monomials appearing in the expression of $f$ with non-zero coefficient. We also let $\vars(f)=\{x_k\;|\: x_k \hbox{ divides some }u\in\supp(f)\}$.
\end{notation}

\begin{definition}
Let $K$ be a field, and $S=K[x_1,\ldots,x_n]$. We say that a polynomial $f \in S$ is \emph{square-free supported} if its support consists of square-free monomials.
\end{definition}

\begin{remark}
Our terminology differs from \cite{BMW}, which uses the term ``square-free''. We decided to add the word ``supported'' to distinguish from polynomials of $S$ without multiple irreducible factors, which are also usually referred as square-free.
\end{remark}

We will now prove several useful properties of square-free supported polynomials. The first is that this property is stable upon a change of variables: 
if $f \in K[x_1, \ldots, x_]$ is a square-free polynomial and $y_i = x_i + a_i$, with $a_i \in K$ for $1 \leq i \leq n$, then $f$ remains square-free supported also as an element of $K[y_1, \ldots, y_n]$ \cite[Lemma~3.2]{BMW}.

\begin{lemma}\label{LemmaFactors}
 Let $K$ be a field, $S=K[x_1,\ldots,x_n]$ and $f \in S$ be a square-free supported polynomial with irreducible factors $f_1,\ldots,f_t$.
 Then
 \begin{enumerate}
 \item  $f_1,\ldots,f_t$ are square-free supported polynomials, and 
 \item $\vars(f_i)\cap\vars(f_j)=\varnothing$ for every $i\neq j$.
 \end{enumerate}
\end{lemma}
In particular, $S/(f_1, \ldots, f_t)$ is a complete intersection.
\begin{proof}
We show  $(1)$ by contradiction. Suppose that $f_i$ is not square-free supported for some $i$. Let $<$ denote a  lexicographical monomial order such that the greatest variable has a degree bigger than $1$ in a monomial in $f_i$. Then, the leading monomial of $f$ is divisible by the leading monomial of $f_i$. Hence, $f$ is not square-free supported, a contradiction. 

For $(2)$, by contradiction assume that there exist $j \in \{1,\ldots,n\}$ and $i_1 \ne i_2$ such that $x_j \in \vars(f_{i_1}) \cap \vars(f_{i_2})$. Let $\deg_j(-)$ denote the degree in the variable $x_j$. We have $\deg_j(f) = \sum_{s=1}^t \deg_j(f_s) \geq \deg_j(f_{i_1}) + \deg_j(f_{i_2}) \geq 2$, so that $x_j^2$ divides a monomial in the support of $f$. A contradiction.
% We show  $(2)$ by contradiction. Suppose that $x_k$ divides a monomial in the support of $f_i$ and $f_j$ with $i\neq j$. Let $<$ denote a  lexicographical monomial order such that $x_k$ is the greatest variable. Then, $x_k^2$ divides the leading term of $f$, a contradiction.
 \end{proof}

\subsection{Geometric irreducibiity}
We will now prove that a square-free supported polynomial $f \in S=K[x_1,\ldots,x_n]$ is geometrically irreducible, that is, its image in $S_L \coloneqq S \otimes_K L \cong L[x_1,\ldots,x_n]$ is irreducible for every field extension $L$ of $K$.

Let $A$ be a UFD, and $a,b \in A$ with $a \ne 0$. We let $\gcd_A(a,b)$ be the greatest common divisors of $a$ and $b$ in $A$, which is defined up to a unit of $A$. In particular, $\gcd_A(a,b)=1$ means that the only common divisors of $a$ and $b$ are the units of $A$. We also let $\lcm_A(a,b)$ be the least common multiple of $a$ and $b$; recall that $\lcm_A(a,b)\gcd_A(a,b) = ab$.

%\begin{remark}
%If $A$ is a UFD, and $0 \ne a,b \in A$,  the greatest common divisor of $a$ and $b$ in $A$,  $\gcd_A(a,b)$, is only defined up to an invertible element. 
%\end{remark}
\begin{lemma} \label{lemma1} Let $A$ be a domain, and $f \in A[x]$ be a polynomial of degree $1$. If we write $f=ax+b$ with $a,b \in A$, then the following are equivalent:
\begin{enumerate}
\item $f$ is reducible;
\item there exists $c \in A$ not invertible such that $(a,b) \subseteq (c)$.
\end{enumerate}
Moreover, if $A$ is a UFD, then the two conditions are also equivalent to $\gcd_A(a,b) \ne 1$.
\end{lemma}
\begin{proof}
The last statement is equivalent to (2) if $A$ is a UFD. We now prove the equivalence of (1) and (2). First, if $f$ is reducible, then by degree reasons we must have that $f=(a_1x+b_1)c$ with $a_1,b_1,c\in A$ and $c$ not invertible. We conclude that $a=a_1c$ and $b=b_1c$, so that $(a,b) \subseteq (c)$ with $c$ not invertible.

Now assume (2). We can write $a=a_1c$ and $b=b_1c$, so that $f=(a_1x+b_1)c$ is a factorization of $f$, which is then reducible.
\end{proof}

%\begin{corollary} \label{coroll}
%If $A$ is a UFD, and $f = ax+b$ is a polynomial of degree one in $A[x]$, then $f$ is irreducible if and only if $\gcd(a,b)=1$.
%\end{corollary}

\begin{lemma} \label{lemma 2} Let $K$ be a field, and $S=K[x_1,\ldots,x_n]$. Let $K \subseteq L$ be a field extension, and $S_L=L[x_1,\ldots,x_n]$. Given polynomials $0\ne f,g \in S$, we have that $\gcd_S(f,g)$ and $\gcd_{S_L}(f,g)$ generate the same ideal in $S_L$. In particular, $\gcd_S(f,g)=1$ if and only if $\gcd_{S_L}(f,g) = 1$.
\end{lemma}
\begin{proof}
First, note that $(fS) \cap (gS) = (\lcm_S(f,g))$. On the other hand, since $S \to S_L$ is faithfully flat, we have that 
\[
(\lcm_S(f,g))S_L = ((f) \cap (g))S_L = (fS_L) \cap (gS_L) = (\lcm_{S_L}(f,g)).
\]
This means that there exists $\lambda \in L \smallsetminus \{0\}$ such that $\lcm_S(f,g) = \lambda \lcm_{S_L}(f,g)$. But then $\gcd_{S_L}(f,g) = \displaystyle \frac{fg}{\lcm_{S_L}(f,g)} = \lambda \frac{fg}{\lcm_S(f,g)}$ $= \lambda \gcd_S(f,g)$, as desired.
\end{proof}

\begin{theorem} \label{thm irred} Let $K$ be a field, and $S=K[x_1,\ldots,x_n]$. Let $f \in T = S[x]$ be a polynomial of degree one in the variable $x$, that is, $f=gx+h$ for some $g,h \in S$ with $g \ne 0$. Let $K \subseteq L$ be a field extension, $S_L=L[x_1,\ldots,x_n]$ and $P=S_L[x]$. If $f$ is irreducible in $T$, then it is irreducible in $P$.
\end{theorem}
\begin{proof}
By Lemma \ref{lemma1} we have that $f$ is irreducible in $P$ if and only if $\gcd_{S_L}(g,h)=1$. Our assumptions and Lemma \ref{lemma1} guarantee that $\gcd_S(g,h)=1$, and by Lemma \ref{lemma 2} this is preserved when extending the scalars to $S_L$.
\end{proof}

%Let $f \in S=K[x_1,\ldots,x_n]$ be a square-free irreducible polynomial. Let $K \subseteq L$ be any field extension. Then $f$ is square-free and irreducible in $L[x_1,\ldots,x_n]$ as well.
%\end{lemma}
%\begin{proof}
%We proceed by induction on $n \geq 1$. The base case is clear since $f$ is forced to have degree one in this case. Assume $n>1$. With $x=x_n$ write $f=ax+b$ with $a,b \in T=K[x_1,\ldots,x_{n-1}]$ square-free polynomials. By Corollary \ref{coroll} we have that $f$ is irreducible in $L[x_1,\ldots,x_n]$ if and only if $\gcd(a,b) = 1$ in $P=L[x_1,\ldots,x_{n-1}]$. On the other hand, our assumption is that $\gcd(a,b)=1$ in $T$, again by Corollary \ref{coroll}. If we view $a,b$ as polynomials in $T$, then by Lemma \ref{lemma1} their irreducible factors are square-free as well, and by induction they stay irreducible in $P$. If, by way of contradiction, we had $\gcd(a,b)\ne 1$ in $P$, then two irreducible factors $a_1$ of $a$ and $b_1$ of $b$ in $T$ would become equal in $P$. This means that there exists $\lambda \in L$ such that $\lambda a_1 = b_1$. This equality forces $\lambda \in K$, and gives the desired contradiction because it implies that $a$ and $b$ have a common irreducible factor in $T$.
%\end{proof}

\begin{corollary}\label{cor sqfree irred}
Let $K$ be a field, and $S=K[x_1,\ldots,x_n]$. If $f \in S$ is irreducible and square-free supported, then $S/fS$ is geometrically irreducible.
\end{corollary}
\begin{proof}
By definition, a square-free supported polynomial has degree at most one in {\it any} of the variables appearing in its support. Thus the claim follows from Theorem~\ref{thm irred}.
\end{proof}

\subsection{Positive characteristic methods} Let $R$ be a Noetherian ring of prime characteristic $p$. For $e>0$, the $e$th iteration of the Frobenius map $F^e \colon R \to R$, defined as $F^e(f) = f^{p^e}$, induces an $R$-module structure on the ring itself given by restriction of scalars. If we denote by $F^e_*(R)$ this $R$-module structure, and we denote by $F^e_*(r)$ an element of such a module, we have that $F^e_*(f_1)+F^e_*(f_2) = F^e_*(f_1+f_2)$ and $f_1F^e_*(f_2) = F^e_*(f_1^{p^e}f_2)$ for all $f_1,f_2 \in R$.

\begin{definition}
Let $R$ be a Noetherian ring of prime characteristic $p$. We say that $R$ is:
\begin{enumerate}
\item \emph{F-finite} if $F^e_* (R)$ is a finitely generated $R$-module for some (equivalently, all) $e>0$.
\item \emph{F-split} if there exists an $R$-linear map $\phi\colon F_*(R)\to R$
such that $F_*(1)\mapsto 1$.
\item \emph{F-rational} if every parameter ideal of $R$ is tightly closed. Here, $I=(f_1, \ldots, f_t)$ is called a parameter ideal if, for any prime ideal $I \subseteq \p$, the images of $f_1, \ldots, f_n$ in $R_\p$ form a part of a system of parameters in $R_\p$. 
\item \emph{strongly F-regular} if it is F-finite, and for every $c\in R\smallsetminus\bigcup_{\p\in\Min(R)} \p$ there exists $e\in \ZZ_{>0}$ and an $R$-linear map $\phi\colon F^e_*(R)\to R$ such that  $F^e_* (c)\mapsto 1$.
\end{enumerate}
\end{definition}

%\begin{definition}
%Let $R$ be a Noetherian ring of prime characteristic $p$.
%We say that $R$ is F-rational if every parameter ideal of $R$ is tightly closed, 
%where an ideal $I = (x_1, \ldots, x_n)$ is parameter if, for any prime ideal $I \subseteq \p$, 
%the images $x_1, \ldots, x_n$ in $R_\p$ form a part of a system of parameters in $R_\p$. 
%\end{definition}

\begin{remark}
Strongly F-regular rings are clearly F-split. It is known that a strongly F-regular ring is F-rational, a Gorenstein F-finite F-rational ring is strongly F-regular, and if $R$ is F-rational and local or standard graded, then it is a normal domain \cite[Theorems 4.2 \& 5.5]{HH}.
\end{remark}

We now recall Fedder's criterion to test whether a ring is F-split.

\begin{theorem}[{\cite[{Lemma 1.6 \& Corollary 1.7}]{Fedder}}]
Let $S$ be a Noetherian F-finite regular ring of prime characteristic $p>0$. Let $I\subseteq S$ be an ideal, and $R=S/I$. There exists an isomorphism of $F_*(R)$-modules 
\[
\displaystyle \Hom_R(F^e_*(R),R) \cong \frac{IF^e_*(S):_{F^e_*(S)} F^e_*(I)}{IF^e_*(S)}.
\]
Furthermore, if $(S,\m)$ is a local ring, then $R$ is F-split if and only if $I^{[p^e]}:_SI\not\subseteq \m^{[p^e]}$.
\end{theorem}

Fedder's criterion is a very effective tool to establish whether a quotient of a regular local ring is F-split. 
In our work, we will use its version for strong F-regularity, due to Glassbrenner, which improves on the Hochster--Huneke criterion \cite[Theorem 3.3]{HoHu}.

\begin{theorem}[{\cite[Theorem 2.3]{Glassbrenner}}] \label{Glassbrenner}
Let $(S,\m)$ be an F-finite regular local ring, and $I \subseteq S$ be an ideal. Let $R=S/I$. Then $R$ is strongly F-regular if and only if 
there for some $f \in S$ whose image in $R$ does not belong to any minimal prime, $R_f$ is strongly F-regular and 
there exists $e>0$ such that $f(I^{[p^e]}:_SI) \not\subseteq \m^{[p^e]}$.
%Given $f \in S$, and $e>0$, there exists an $R$-linear map $\psi\colon F^e_*(R)\to R$ such that $F^e_*(f) \mapsto 1$ if and only if $f(I^{[p^e]}:I) \not\subseteq \m^{[p^e]}$. 
\end{theorem}

We note that the computation is easier for complete intersections.

\begin{corollary}\label{cor: CI}
Let $(S,\m)$ be an F-finite regular local ring of prime characteristic $p>0$, and $f_1,\ldots,f_c \in \m$ be a regular sequence. Let $Q=(f_1,\ldots,f_c)$, $R=S/Q$ and $f=f_1 \cdots f_c$.
Then $R$ is strongly F-regular if and only if 
there is an element $g \in S$ not in any minimal prime of $Q$ such that 
\begin{enumerate}
\item $R_g$ is strongly F-regular, and 
\item $gf^{p^e-1} \notin \m^{[p^e]}$ for some $e > 0$.
\end{enumerate}
 \end{corollary}
\begin{proof}
We have $(Q^{[p^e]}:_SQ) = (f^{p^e-1}) + Q^{[p^e]}$ \cite[Proposition 2.1]{Fedder}.
\end{proof}

 \subsection{Main results} 
 
 Let $K$ be a field, $I \subseteq S=K[x_1,\ldots,x_n]$, and $R=S/I$. We say that $R$ is \emph{geometrically F-rational} if $R_L \coloneqq R \otimes_K L$ is F-rational for every field extension $L$ of $K$. The following is our main result.

\begin{theorem} \label{ThmReducible}
Let $K$ be a field of characteristic $p>0$, and $S=K[x_1,\ldots,x_n]$. Let $f \in S$ be a square-free supported polynomial with irreducible factors $f_1,\ldots,f_t$. If we let $Q=(f_1,\ldots,f_t)$, then $R=S/Q$ is geometrically F-rational.
\end{theorem}
\begin{proof}
%If   $\overline{K}\otimes_K R$ is strongly F-regular then so is $R$, because the latter is a direct summand of the former. In addition, 
Let $L$ be a field extension of $K$. The image of $f_i$ in $S_L \coloneqq S \otimes_K \overline{L}$ is irreducible for every $i$ by Corollary~\ref{cor sqfree irred}, and the image of $f$ is still square-free supported. We note that it suffices to show that $R_{\overline{L}}$ is strongly F-regular. In fact, if that was the case, then $R_{\overline{L}}$ would in particular be F-rational, and since the map $R_L \to R_{\overline{L}}$ is faithfully flat, we would conclude that $R_L$ is F-rational \cite[Proposition A.5]{DattaMurayama}. Thus, replacing $K$ with $\overline{L}$, we may directly assume that $K$ is algebraically closed and show that $R$ is strongly F-regular. Strong F-regularity can be tested locally, so let $\m$ be a maximal ideal of $S$ containing $Q$. We want to show that $R_\m$ is strongly F-regular. Since the base field is now algebraically closed, and a change of coordinates $x_i \mapsto x_i+a_i$ with $a_i \in K$ does not affect the fact that $f$ is square-free supported \cite[Lemma 3.2]{BMW}, we may assume that $\m=(x_1,\ldots,x_n)$. By Lemma \ref{LemmaFactors}, the irreducible factors $f_i$ of $f$ are square-free supported and satisfy $\vars(f_i) \cap \vars(f_j) = \varnothing$ for all $i \ne j$.

We proceed by induction on $n \geq 1$. The base case $n=1$ is immediate because $f$ is forced to be equal to $x_1$. 
Now assume $n>1$. 
If for every $i \in \{1,\ldots,t\}$ the support of $f_i$ contains a monomial of degree one, say a variable $x_{k_i}\in\vars(f_i)$, then we can rewrite $f_i = \lambda x_{k_i} + g_i$ where $\lambda$ is a unit in $R_\m$ and all the monomials in the support of $g_i$ are not divisible by $x_{k_i}$. Thus
$$
R_\m\cong \left( \frac{K[x_1,\ldots,x_n]}{(x_1,\cdots, x_t)}\right)_\m,
$$
where $Q_\m$ maps to $(x_1,\ldots, x_t)_\m$, and the statement follows.

Now suppose that $\supp(f_i) \subseteq \m^2$ for some $i$; we may assume $i=1$. Suppose without loss of generality that $x_n$ divides a monomial in the support of $f_1$. We claim that the Glassbrenner criterion of Corollary~\ref{cor: CI} applies to $g = x_n$. We start by showing that $(R_\m)_{x_n}$ is strongly F-regular. 
Let $T=K(x_n)[x_1,\ldots,x_{n-1}]$, and note that it is an F-finite ring. We also note that $f$ is a square-free supported polynomial ideal in $T$ and that the image of $f_j$ is irreducible in $T$ for every 
$j$, since $T$ is a localization of $S$. For $j=1$ we also have to observe that $f_1$ is not a unit in $T$ because $f_1\in\m^2$ by assumption. We have that $T/QT$ is a strongly F-regular ring by induction.
Since $(R_\m)_{x_n}$ is a localization of $T/QT$, it is strongly F-regular. 

We now verify the second condition. 
Since $f_1$ is irreducible, the set $\mathcal{A}=\{u \in \supp(f_1) \mid x_n$ does not divide $u\}$ is not empty. Choose $u \in \mathcal{A}$ of minimal degree. Then there is a monomial $v \in \supp(f_1^{p-1})$ which is not divisible by $x_n$ and such that $v \notin \m^{[p]}$. Since $\vars(f_i) \cap \vars(f_j) = \varnothing$ for $i \ne j$, and each $f_j$ is square-free supported, there exists a monomial $w \in \supp(f^{p-1})$ which is not divisible by $x_n$ and such that $w \notin \m^{[p]}$. In particular, we have that $x_nf^{p^e-1} \notin \m^{[p^e]}$. 
Since $f_1,\ldots,f_t$ forms a regular sequence in $S_\m$ by Lemma~\ref{lemma1}, $R_\m$ is strongly F-regular by Corollary~\ref{cor: CI}.
 \end{proof}
 
 We are now able to positively answer Question \ref{Quest} without assuming that the base field $K$ is algebraically closed and, even more, obtaining geometric F-rationality.
 
 \begin{corollary} \label{CorPoly}
  Let $K$ be a field of characteristic $p>0$, and $S=K[x_1,\ldots,x_n]$. If $f \in S$ is an irreducible square-free supported polynomial, then $S/(f)$ is geometrically F-rational.
 \end{corollary}

Corollary \ref{CorPoly} implies that the variety defined by an irreducible square-free supported polynomial in characteristic zero has rational singularities. This extends to any field of characteristic zero previous work of Bath, Musta\c{t}\u{a}, and Walther \cite[Theorem 1.1]{BMW}.

\begin{corollary} \label{Char0} Let $K$ be a field of characteristic zero, and $Z \subseteq \mathbb{A}^n_K$ be a hypersurface defined by an irreducible square-free supported polynomial $f \in K[x_1,\ldots,x_n]$. Then $Z$ has rational singularities.
\end{corollary}
\begin{proof}
Let $A$ be a finitely generated $\ZZ$-algebra which contains all coefficients of the polynomial $f$. We can then view $f$ as a polynomial in $T=A[x_1,\ldots,x_n]$, then $T/fT \otimes_A K$ coincides with the original hypersurface.
By Corollary~\ref{cor sqfree irred} $T/fT$ is geometrically irreducible, therefore there exists a dense open subset $U$ of ${\rm Max}\Spec(A)$ such that the image of $f$ in $T_{\kappa(\p)}= \kappa(\p)[x_1,\ldots,x_n]$ is irreducible for every $\p \in U$, with $\kappa(\p) = A/\p$ \cite[Theorem 2.3.6]{HHChar0}. Note that the image of $f$ in $T_{\kappa(\p)}$ is square-free supported. By Corollary \ref{CorPoly} we have that $T_{\kappa(\p)}/(f)$ is F-rational for every $\p \in U$, and thus $T_K/(f)$ has rational singularities by \cite[Theorem~4.3]{Smith}.
\end{proof}

We now present a consequence of Corollary \ref{CorPoly} for polynomials that are not necessarily square-free supported, along the lines of \cite[Theorem 1.3]{BMW}.

\begin{corollary}\label{CorSqfreeModification}
Let $K$ be a field of characteristic $p>0$, and $g, h\in S=K[x_1,\ldots,x_n]$ be square-free supported homogeneous polynomials satisfying $\deg(h)=1 + \deg(g)$. Let $\ell=a_1 x_1 +\cdots+a_n x_n+1$ for some $a_1,\ldots, a_n\in K$. If $g$ is irreducible and does not divide $h$, then $f=g\ell+h$ is geometrically irreducible and $S/(f)$ is geometrically F-rational.
\end{corollary}
\begin{proof}
Let $\widetilde{\ell} = a_1 x_1 +\cdots+a_n x_n +z$ and $\widetilde{f}= g\widetilde{\ell} +h\in \widetilde{S} \coloneqq
K[x_1,\ldots, x_n,z]$. We note that $\widetilde{\ell}$ is the homogenization of $\ell$ and $\widetilde{f}$ is the homogenization of $f$. We consider the change of coordinates $y_i = x_i$ for $1\leq i \leq n$ and $y = \tilde{\ell}$. Since $g$ and $h$ are square-free supported polynomials, $\widetilde{f} = gy+h$ is a square-free supported homogeneous polynomial with respect to the coordinates $y_1,\ldots, y_n,y$. Now let $L$ be a field extension of $K$. Since $\gcd_S(g,h)=1$, it follows from Lemma \ref{lemma1} and Theorem \ref{thm irred} that $\widetilde{f}$ is irreducible in $S_L \coloneqq S \otimes_K L$, and that $\widetilde{R}=\widetilde{S}/(\widetilde{f}) \otimes_K L$ is F-rational by Theorem \ref{CorPoly}. Let $T=\widetilde{S}[z^{-1}] \otimes_K L$, and $\phi\colon T\to T$ be the $L$-algebra automorphism given by $x_i\mapsto x_iz^{-1}$ and $z\mapsto z$. We note that $\phi(f)=\widetilde{f}z^{-\deg(\widetilde{f})}$, so that $\phi((f)T)=(\widetilde{f})T$. Then, $\phi$ induces an isomorphism $T/(f)T \to T/(\widetilde{f})T$, and it follows that $T/(f)T$ is F-rational. Since the map $R_L \to R_L[z,z^{-1}] \cong T/(f)T$ is faithfully flat, we obtain that $R_L$ is also F-rational \cite[Proposition A.5]{DattaMurayama}.
\end{proof}

Finally, we prove a characteristic zero version of Corollary \ref{CorSqfreeModificationRational}. This recovers the aforementioned result by Bath, Musta\c{t}\u{a}, and Walther \cite[Theorem 1.3]{BMW}, but does not require that the base field is algebraically closed.
\begin{corollary}\label{CorSqfreeModificationRational}
Let $K$ be a field of characteristic zero, and $g, h\in S=K[x_1,\ldots,x_n]$ be square-free supported homogeneous polynomials satisfying $\deg(h)=1 + \deg(g)$.
Let $\ell=a_1 x_1 +\cdots+a_n x_n+1$ for some $a_1,\ldots, a_n\in K$. Assume that $g$ is irreducible and does not
divide $h$. If $Z \subseteq \mathbb{A}^n_K$ is the hypersurface defined by $f=g\ell +h$, then $Z$ has rational singularities.
\end{corollary}
\begin{proof}
There is a finitely generated $\ZZ$-algebra $A$ such that $f$ is defined over $T=A[x_1,\ldots,x_n]$, and its image in $T_K = T \otimes_A K$ coincides with the original polynomial. By Corollary \ref{CorSqfreeModification} we have that $f$ is geometrically irreducible. Thus, there exists a dense open subset $U$ of ${\rm Max}\Spec(A)$ such that the image of $f$ in $T_{\kappa(\p)}= \kappa(\p)[x_1,\ldots,x_n]$ is irreducible for every $\p \in U$ \cite[Theorem 2.3.6]{HHChar0}. It follows from Corollary \ref{CorSqfreeModification} that $T_{\kappa(\p)}/(f)$ is F-rational for every $\p \in U$, and we conclude that $Z$ has rational singularities by \cite[Theorem~4.3]{Smith}.
\end{proof}

\section{Formulas for the defect of the F-pure threshold} \label{SecDFPT}

We start by recalling the relevant definitions.

\begin{definition}[\cite{TW2004}]\label{DefFptGlobalIdeal}
Let $R$ be an F-finite F-split ring and $J\subseteq R$ be an ideal.
\begin{itemize}
\item For $\lambda \in \R_{\geq 0}$ we say that the pair $(R,J^\lambda)$ is F-pure if, for all $e\gg 0$, there exists $f\in J^{\lfloor \lambda(p^e-1) \rfloor }$ and a map 
$\phi \in \Hom_R(F^e_*(R),R)$ such that
$\phi(F^e_*(f))=1$.
\item If $(R, \m)$ is local, the F-pure threshold of $R$ is defined as
$$
\fpt(R)=\sup\{ \lambda\in \mathbb{R}_{\geq 0}\; |\; (R,\m^\lambda) \hbox{ is F-pure}\}.
$$
\end{itemize}
\end{definition}

\begin{definition}[{\cite{DSNBS}}]\label{DefDFPT}
If $(R,\m)$ is an F-finite F-split local ring, the defect of the F-pure threshold is defined as $\dfpt(R)=\dim(R)-\fpt(R)$. 
\end{definition}

\begin{definition}
Let $(R,\m)$ be a local ring of dimension $d$. The \emph{Hilbert--Samuel multiplicity} of $R$ is
\[
\e(R) = \lim\limits_{n \to \infty} \frac{d!\ell_R(R/\m^n)}{n^d}.
\]
\end{definition}

We recall that, if $X$ is a topological space, a function $\varphi\colon X \to \R$ is \emph{upper semi-continuous} if, for every $\lambda \in \R$, the set $\{x \in X \mid \varphi(x) < \lambda\}$ is open. 

\begin{remark}
Since $\Spec(R)$ is quasi-compact, any upper semi-continuous function on $\Spec(R)$ has a maximum. For example, if $R$ is an F-finite F-split locally equidimensional ring then then we have upper semi-continuous functions 
$\Spec(R) \ni \p \mapsto \dfpt(R_\p)$ \cite[Theorem 4.3]{DSNBS}
and  $\Spec(R) \ni \p \mapsto \e(R_\p)$ (e.g., see  \cite[3.1.6]{Smirnov} and note that F-finite rings are excellent \cite{Kunz}).
Thus, we can define 
\[
\dfpt(R) \coloneqq \max\{\dfpt(R_\p) \mid \p \in \Spec(R)\}
\text{ and } \e(R) \coloneqq \max\{\e(R_\p) \mid \p \in \Spec(R)\}.
\]
\end{remark}

 \begin{theorem} \label{ThmDfptOne}
 Let $K$ be an algebraically closed field of characteristic $p>0$, let $S=K[x_1,\ldots,x_n]$, and $f \in S$ be a square-free supported polynomial. Let $R=S/(f)$. Then $\dfpt(R_\m)=
\e(R_\m)-1$ for every maximal ideal $\m$ of $R$. As a consequence, $\dfpt(R) = \e(R)-1$.
\end{theorem}
\begin{proof}
First, fix a maximal ideal $\m$ of $S$ containing $f$. Since $K$ is algebraically closed, by changing coordinates it suffices to assume that $\m=(x_1,\ldots,x_n)$  \cite[Lemma 3.2]{BMW}. We note that the lowest degree part of $f$, $\init(f)$, is a square-free supported homogeneous polynomial. 
By Fedder's criterion $S/\init(f)$ is F-split, so by \cite[Theorem 5.8]{DSNBS} we have that 
\[
\dfpt(R_\m)=\dfpt(S/\init(f)) = -\max\{j \mid [H^{n-1}_\eta(S/\IN(f))]_j \ne 0\}.
\]
Furthermore,  because $S/\IN(f)$ is a standard graded F-split $K$-algebra, we have that 
\begin{align*}
max\{j \mid [H^{n-1}_\eta(S/\IN(f))]_j \ne 0\} 
=\deg(\IN(f)) - n 
= \ord_\m (f) - n = \e(R_\m) - \dim (R_\m) - 1,
\end{align*}
where $\ord_\m(f) = \max\{t \mid f \in \m^t\}$ is the $\m$-adic order of $f$. The result now follows.
\end{proof}

 \begin{corollary} \label{CorDfptMany}
 Let $K$ be an algebraically closed  field of characteristic $p>0$, let $S=K[x_1,\ldots,x_n]$ and $f \in S$ be a square-free supported polynomial with irreducible factors $f_1,\ldots,f_t$. Let $Q=(f_1,\ldots,f_t)$, and $R=S/Q$. Then $\dfpt(R_\m)=\e((S/(f))_\m)-t$ for every maximal ideal $\m$ of $S$ which contains $Q$. Moreover, we have that $\dfpt(R) = \e(S/(f))-t$.
\end{corollary}
\begin{proof}
Since $K$ is algebraically closed, and thanks to \cite[Lemma 3.2]{BMW}, after performing a change of coordinates we may assume that $\m=(x_1,\ldots,x_n)$. By Lemma \ref{LemmaFactors}, we have that 
\[
R\cong \frac{K[\vars(f_1)]}{(f_1)}\bigotimes_K \cdots\bigotimes_K \frac{K[\vars(f_t)]}{(f_t)}\bigotimes_K K[A],
\]
where $A=\{x_1,\ldots,x_n\}\smallsetminus \bigcup^t_{i=1}\vars(f_i)$.
Let $\m_i= (\vars(f_1))$ and $\eta=(A)$.
Then
\begin{align*}
\dfpt(R_\m)& =\dfpt\left(\left(\frac{K[\vars(f_1)]}{(f_1)}\right)_{\m_1}\right)+\cdots+\dfpt\left(\left(\frac{K[\vars(f_t)]}{(f_t)}\right)_{\m_t}\right) + \dfpt(K[A]_\eta) \\
&=\dfpt\left(\left(\frac{K[\vars(f_1)]}{(f_1)}\right)_{\m_1}\right)+\cdots+\dfpt\left(\left(\frac{K[\vars(f_t)]}{(f_t)}\right)_{\m_t}\right)\\
&=\e\left( \left(\frac{K[\vars(f_1)]}{(f_1)}\right)_{\m_1}\right)+\cdots+\e\left( \left(\frac{K[\vars(f_t)]}{(f_t)}\right)_{\m_t}\right)-t \\
& = \e((S/(f_1))_\m) + \cdots + \e((S/(f_t))_\m) - t \\
&=\e((S/(f))_\m)-t,
\end{align*}
where the first equality follows from a formula of the defect of the F-pure threshold for tensor products of $K$-algebras \cite[Corollary 3.23]{DSNBS}, the second from the fact that $K[A]$ is regular, the third from Theorem \ref{ThmDfptOne}, and the last one from the associativity formula for multiplicity. 
Since the formula holds for all maximal ideals containing $Q$, it follows that 
\[
\dfpt(R) = 
\max\{\e((S/(f))_\m) \mid \m \ni Q\} -t 
\leq \max\{\e((S/(f))_\m) \mid \m \ni f\} - t = \e(S/(f))-t.
\]
Thus, it suffices to show that $\e(S/(f))$ is achieved at a maximal ideal $\m$ containing $Q$. 

We argue by contradiction.
Let $\m$ be a maximal ideal such that  $\e((S/(f))_\m) = e(S/(f))$, and assume without loss of generality that $f_1 \notin \m$. Then $\e((S/(f_1))_\m) = 0$ and thus $\e((S/(f))_\m) = \e((S/(g))_\m)$ with $g=f/f_1$. Let $\n$ be a maximal ideal of $S$ containing $f_1$. Let $B=\vars(f_1)$ and $C=\{x_1,\ldots,x_n\} \smallsetminus B$. 
By Lemma~\ref{lemma1} $\a=(\m \cap K[C])S + (\n \cap K[B])S$
is a maximal ideal. Then with this choice we have 
\begin{align*}
\e((S/(f))_\a) &  = \e((S/(g))_\a) + \e((S/(f_1))_\a) \\ 
& = \e((S/(g))_\m) + \e((S/(f_1))_\n) \\
& = \e((S/(f))_\m) + \e((S/(f_1))_\n) > \e((S/(f))_\m).
\end{align*}
This gives the desired contradiction, and completes the proof.

%Let $\m$ be a maximal ideal that maximizes the multiplicity of $S/(f)$, and assume by the way of contradiction that $f_i \notin \m$ for some $i$. Then $\e((S/(f_i))_\m) = 0$ and thus $\e((S/(f))_\m) = \e((S/(g))_\m)$ with $g=f/f_i$. Let $\n$ be a maximal ideal of $S$ containing $f_i$. Let $B=\vars(f_i)$ and $C=\{x_1,\ldots,x_n\} \smallsetminus B$. Finally, consider the maximal ideal
%\[
%\a=(\m \cap K[C])S + (\n \cap K[B])S.
%\]
%With this choice we have 
%\begin{align*}
%\e((S/(f))_\a) &  = \e((S/(g))_\a) + \e((S/(f_i))_\a) \\ 
%& = \e((S/(g))_\m) + \e((S/(f_i))_\n) \\
%& = \e((S/(f))_\m) + \e((S/(f_i))_\n) > \e((S/(f))_\m).
%\end{align*}
%This gives the desired contradiction, and completes the proof.
\end{proof}

\begin{corollary}
Let $K$ be an algebraically closed field, and $g, h\in S=K[x_1,\ldots,x_n]$ be square-free supported homogeneous polynomials satisfying $\deg(h)=1 + \deg(g)$.
Let $\ell=a_1 x_1 +\cdots+a_n x_n+1$ for some $a_1,\ldots, a_n\in K$. If $g$ is irreducible and does not
divide $h$, then 
\[
\dfpt(S/(g\ell + h))=\e(S/(g\ell + h))-1.
\]
\end{corollary}
\begin{proof}
Fix $a_1, \ldots, a_n \in K$ and let $f=g\ell+h$ and $R=S/(f)$. Following the proof of Corollary~\ref{CorSqfreeModification},  let $\widetilde{\ell} = a_1 x_1 +\cdots+a_n x_n +z$ and $\widetilde{f}= g\widetilde{\ell} +h\in \widetilde{S}=K[x_1,\ldots, x_n,z]$. Recall that $\widetilde{f}$ is square-free supported and homogeneous with respect to the variables
$y_1,\ldots, y_n,y$, so if we let $\widetilde{R} = \widetilde{S}/(\widetilde{f})$, then it follows that $\dfpt(\widetilde{R_\n})=\e(\widetilde{R_\n})-1$
for every maximal ideal $\n\subseteq \widetilde{R}$ by Theorem~\ref{ThmDfptOne}. 

Let $\m$ be a maximal ideal of $R$, $\n= (\m, z-1)R[z,z^{-1}]$, $B = (R[z, z^{-1}])_\n$.
Note that $R_\m \to B \cong (R\otimes_K K[z])_\n$ is a flat local map with regular closed fiber, so $\dfpt(R_\m)=\dfpt(B)$ \cite[Proposition 5.12]{DSNBS}. 
Furthermore, $\e(R_\m) = \e(B)$ since 
\[
\dim_K \left(\frac{\n^kB}{\n^{k+1}B}\right)
= \dim_K \left( \frac{\sum_{a = 0}^k \m^{n-a} (z-1)^{a}}{\sum_{a = 0}^{k+1} \m^{n-a} (z-1)^{a}}\right)
= \sum_{i = 0}^k \dim_K \left(\frac{\m^i}{\m^{i+1}}\right).
\]
Let $T=\widetilde{S}[z^{-1}]$, and let $\phi\colon T \to T$ be the $K$-algebra automorphism defined by $x_i \mapsto x_iz^{-1}$ and $z \mapsto z$. As noted in the proof of Corollary \ref{CorSqfreeModification}, this map induces an isomorphism $B \to \widetilde{R}_{\n'}$ with $\n'=\phi(\n)$, and thus 
\[
\dfpt(B)=\dfpt(\widetilde{R}_{\n'})=\e(\widetilde{R}_{\n'})-1=\e(B)-1.
\]
It follows that $\dfpt(R_\m) = \e(R_\m) - 1$, and because $\m$ was any maximal ideal of $R$ we also obtain $\dfpt(R) = \e(R)-1$.
\end{proof}
%{\color{blue}
%By Equation (\ref{Eq dfpt}), we conclude that
%\[
%\dfpt(R_\m)=\dfpt(R[z,z^{-1}]_\n)=\e(R[z,z^{-1}]_\n)-1=\e(R_\m)-1. \qedhere
%\]
%}

\bibliographystyle{alpha}
\bibliography{References}

\end{document}